\documentclass[11pt,english]{article}

\usepackage[T1]{fontenc}
\usepackage[latin9]{inputenc}
\usepackage{geometry}
\usepackage{babel}
\usepackage{amsmath}
\usepackage{amsthm}
\usepackage{amssymb}
\usepackage[unicode=true,pdfusetitle,
 bookmarks=true,bookmarksnumbered=false,bookmarksopen=false,
 breaklinks=false,pdfborder={0 0 0},pdfborderstyle={},backref=false,colorlinks=false]
 {hyperref}

\makeatletter

\usepackage[nameinlink,capitalise,noabbrev]{cleveref}

\hypersetup{%
    bookmarksnumbered, bookmarksopen=true, bookmarksopenlevel=1,%
}

\theoremstyle{plain}
\newtheorem{thm}{Theorem}[section]
\crefname{thm}{Theorem}{Theorems}
\theoremstyle{plain}
\newtheorem{lem}[thm]{Lemma}
\crefname{lem}{Lemma}{Lemmas}
\theoremstyle{plain}

\theoremstyle{plain}
\newtheorem*{claim*}{Claim}
\crefname{claim}{Claim}{Claims}
\theoremstyle{definition}
\newtheorem{defn}[thm]{Definition}
\theoremstyle{plain}
\newtheorem{conjecture}[thm]{Conjecture}
\theoremstyle{plain}

\theoremstyle{definition}

\theoremstyle{definition}

\theoremstyle{plain}
\newtheorem{claim}[thm]{Claim}

\crefformat{equation}{#2(#1)#3}
\crefname{appsec}{Appendix}{Appendices}

\date{}

\let\originalleft\left
\let\originalright\right
\renewcommand{\left}{\mathopen{}\mathclose\bgroup\originalleft}
\renewcommand{\right}{\aftergroup\egroup\originalright}

\usepackage{verbatim}

\makeatletter
\renewcommand*{\UrlTildeSpecial}{%
  \do\~{%
    \mbox{%
      \fontfamily{ptm}\selectfont
      \textasciitilde
    }%
  }%
}%
\let\Url@force@Tilde\UrlTildeSpecial
\makeatother

\usepackage{enumitem}

\makeatother

\begin{document}

\title{On Kahn's basis conjecture}

\author{Matija Buci\'c\thanks{Department of Mathematics, ETH, Z\"urich, Switzerland. Email: \href{mailto:matija.bucic@math.ethz.ch} {\nolinkurl{matija.bucic@math.ethz.ch}}.}\and
Matthew Kwan\thanks{Department of Mathematics, Stanford University, Stanford, CA 94305. Email: \href{mailto:mattkwan@stanford.edu} {\nolinkurl{mattkwan@stanford.edu}}. This research was done in part while the author was working at ETH Zurich, and is supported in part by SNSF project 178493.}\and
Alexey Pokrovskiy\thanks{Department of Economics, Mathematics and Statistics, Birkbeck, University of London. Email:
\href{mailto:Dr.Alexey.Pokrovskiy@gmail.com}{\nolinkurl{Dr.Alexey.Pokrovskiy@gmail.com}}.}\and Benny Sudakov\thanks{Department of Mathematics, ETH, Z\"urich, Switzerland. Email:
\href{mailto:benjamin.sudakov@math.ethz.ch} {\nolinkurl{benjamin.sudakov@math.ethz.ch}}.
Research supported in part by SNSF grant 200021-175573.}}

\maketitle
\global\long\def\E{\mathbb{E}}
\global\long\def\Var{\operatorname{Var}}
\global\long\def\S{\mathcal{S}}

\begin{abstract}
In 1991, Kahn made the following conjecture. For any $n$-dimensional vector space $V$ and any $n\times n$ array of $n^2$ bases of $V$, it is possible to choose a representative vector from each of these bases in such a way that the representatives from each row form a basis and the representatives from each column also form a basis. Rota's basis conjecture can be viewed as a special case of Kahn's conjecture, where for each column, all the bases in that column are the same. Recently the authors showed that in the setting of Rota's basis conjecture it is possible to find suitable representatives in $\left(1/2-o\left(1\right)\right)n$ of the rows. In this companion note we give a slight modification of our arguments which generalises this result to the setting of Kahn's conjecture. Our results also apply to the more
general setting of matroids.
\end{abstract}

\section{Introduction}

This note should be considered as a companion note to the paper \cite{us}. Here we fill in the details of how to modify our proof given in \cite{us} to also apply to the setting of Kahn's conjecture \cite{HR94}. The proof presented here is mostly self-contained but we refer the interested reader to \cite {us} for motivation and a more detailed take on the proof.
 
\emph{Matroids} are objects that
abstract the combinatorial properties of linear independence in vector
spaces. Specifically, a finite matroid $M=\left(E,\mathcal{I}\right)$
consists of a finite ground set $E$ (whose elements may be thought
of as vectors in a vector space), and a collection $\mathcal{I}$
of subsets of $E$, called independent sets. The defining properties
of a matroid are that:
\begin{itemize}

\item{the empty set is independent (that is, $\emptyset\in\mathcal{I}$);}

\item{subsets of independent sets are independent (that is, if $A'\subseteq A\subseteq E$
and $A\in\mathcal{I}$, then $A'\in\mathcal{I}'$);}

\item{if $A$ and $B$ are independent sets, and $\left|A\right|>\left|B\right|$,
then an independent set can be constructed by adding an element of
$A$ to $B$ (that is, there is $a\in A\backslash B$ such that $B\cup\left\{ a\right\} \in\mathcal{I}$).
This final property is called the \emph{augmentation property}.}

\end{itemize}

Observe that any finite set of elements in a vector space (over any
field) naturally gives rise to a matroid, though not all matroids
arise this way. A \emph{basis }in a matroid $M$ is a maximal independent
set. By the augmentation property, all bases have the same size, and
this common size is called the \emph{rank }of $M$. 
The following conjecture is the natural matroid generalisation of Kahn's original conjecture.

\begin{conjecture}
\label{conj:kahn}Given a rank-$n$ matroid and bases $B_{i,j}$ for
each $1\le i,j\le n$, there exist representatives $b_{i,j}\in B_{i,j}$
such that each $\left\{ b_{1,j},\dots,b_{n,j}\right\} $ and each
$\left\{ b_{i,1},\dots,b_{i,n}\right\} $ are bases.
\end{conjecture}

\begin{thm}
\label{thm:1/2-kahn}
For any $\varepsilon>0$ the following holds for sufficiently large $n$. Given a rank-$n$ matroid and bases $B_{i,j}$ for
each $1\le i\le n$ and $1\le j \le f=(1-\varepsilon)n/2$, there exist representatives $b_{i,j}\in B_{i,j}$ and $L\subseteq \{1,\dots,f\}$
such that each $\{b_{i,j}:i\in L\}$ is independent, and such that
$\{b_{i,1},\dots,b_{i,n}\}$ is a basis for any $i \in L$ and $|L| \ge (1/2-\varepsilon)n$.
\end{thm}

Note that if we are in the setting of \cref{conj:kahn} where bases are given for all $1 \le i,j\le n$ then the above theorem allows us to choose roughly which rows we would like to find our bases in.


\noindent{\bf Notation.} We will frequently want to denote the result of adding and removing single elements from a set. For a set $S$ and some $x\notin S$, $y\in S$, we write $S+x$ to mean $S\cup \{x\}$, and we write $S-y$ to mean $S\setminus \{y\}$.

\section{Finding many row bases}

In this section we prove \cref{thm:1/2-kahn}.

Let $t=f-\varepsilon n/2=(1/2-\varepsilon)n.$ Let $T$ be an $f \times n$ table partially filled with matroid elements coming from the associated bases. We denote by $T(i,j)\in B_{i,j}$ the element in cell $(i,j)$, with $T(i,j)=\emptyset$ denoting that it is empty. We start with $T$ empty and in each iteration increase the number of values entered in $T$ while preserving the independence of the elements in each row and column. We denote by $S_i$ the set of all current entries of $T$ in row $i$, and by $C_j$ the set of all current entries in column $j$.

\subsection{\label{subsec:simple-swaps}Simple swaps}

Let $U$ denote the set of currently non-empty positions in $T$. Our objective is to increase the size of $U$ for as long as we don't have $t$ full rows. If $T(i,j)=\emptyset$ and there is an $x\in B_{i,j}$ independent to $S_i$ and to $C_j$ then we can let $T(i,j)=x.$ 

We will want much more freedom than this: we also want to consider those elements
that can be added to $T$ after making a small change to the current entries of $T$. This
motivates the following definition.
\begin{defn}
\label{def:addable}Let $T(i,b)=\emptyset$. Say an element $x \in B_{i,c}$ is $\left(i,b\right)$-\emph{addable
}if either
\begin{itemize}
\item $T(i,c)=\emptyset$ and both $S_i+x$ and $C_c+x$ are independent, or;
\item $T(i,c)=x'$ and there is $y\in B_{i,b}$ such that $C_b+y$ and $S_i-x'+y+x$ are independent.
\end{itemize}
In the second case, we say that $S_i-x'+y$
is the result of applying a \emph{simple swap }to $S_i$, and we say
$y$ is a \emph{witness} for the $\left(i,b\right)$-addability
of $x$. In the first case, there is no witness for
the addability of $x$.
\end{defn}

Note that if for some $S_i$ missing an element in column $b$ there is an $\left(i,b\right)$-addable element $x\in B_{i,c}$ which is also independent from $C_c$, then we can increase the size of $U$ by setting $T(i,c)=x$, possibly after applying a simple swap to $S_i$. This unfortunately might not always be possible, but a simple swap still allows us to change which column of $T$ is missing an entry in row $i$.

With this in mind, we study which elements of $S$ can be used in a simple swap.
\begin{defn}
Let $T(i,b)=\emptyset$. We say that a column $c$ with $T(i,c)\neq \emptyset$ is \emph{$\left(i,b\right)$-swappable}
if there is a simple swap making $T(i,c)=\emptyset$ and $T(i,b)=y$ in such a way that $C_b+y$ and $S_i+y-x$ are independent. We say that $y$ is a witness for
the \emph{$\left(i,b\right)$}-swappability\emph{ }of $c$.
\end{defn}

(Basically, a column is \emph{$\left(i,b\right)$}-swappable if
in row $i$ we can add an element to column $b$ and remove one from column $c$ while preserving independence of $C_b$ and $S_i$.) 
\begin{claim}
\label{claim:many-good}Suppose $T(i,b)=\emptyset$. We can either directly increase the size of $U$ or there are at least $n-\left|C_b\right|$ columns which are $\left(i,b\right)$-swappable.
\end{claim}

\begin{proof}
By the augmentation property there is a set $I\subseteq B_{i,b}$ of at least $n-|C_b|$ elements which are independent to $C_b.$ If $|S_i|<n-|C_b|$ then again by the augmentation property applied to $I$ and $S_i$ there is an element $x\in B_{i,b}$ which is also independent to $S_i$. Therefore, setting $T(i,b)=x$ we increase the size of $U.$
Suppose now that $|S_i|\ge n-|C_b|$. Let $S\subseteq S_i$ be the set of all elements of
$S_i$ which are in an $\left(i,b\right)$-swappable column. Let us assume for the sake of contradiction that $\left|S\right|<n-|C_b|$.
This implies $\left|S\right|<\left|S_i\right|$. As $|I|\ge n-|C_b|$ there is a $y\in I$ such that $S+y$ is independent. Using the augmentation
property, we can add $\left|S_i-S\right|-1$ elements of $S_i-S$ to
$S+y$ to obtain an independent set $S_i+y-x'$
for some $x'\in S-S'$ with $T(i,c)=x'$. But this means $c$ is $\left(i,b\right)$-swappable, which is a contradiction.
\end{proof}

Now we show that if $c$ is \emph{$\left(i,b\right)$}-swappable, then all elements of $B_{i,c}$ which are independent to $S_i$ are \emph{$\left(i,b\right)$}-addable,
unless there is an \emph{$\left(i,b\right)$}-addable element not
in $F$.
\begin{claim}
\label{claim:add-if-good}Suppose $T(i,b)=\emptyset$, and let column $c$ be $\left(i,b\right)$-swappable with witness $y$. Either
$S_i+y$ is independent, or for any $x\in B_{i,c}$
independent of $S_i$, $x$ is $\left(i,b\right)$-addable.
\end{claim}

\begin{proof}
Let $T(i,c)=x'$. Consider some $x\in B_{i,c}$ independent to $S_i$. Let $I=S_i+x$ and $J=S_i+y-x'$. Note that $J$ is independent by the $(i,b)$-swappability of $c$ so by the augmentation property, there is an element of $I\backslash J$ that is independent of $J$; this element is either $x'$ or $x$. In the former case $S_i+y$
is independent. In the latter case, $S_i+y-x'+x$ is independent, showing that $x$ is $\left(i,b\right)$-addable. 
\end{proof}

Note that if the first case of \cref{claim:add-if-good} occurs, it means we can set $T(i,b)=y$ and increase the size of $U$, while preserving the independence conditions.
The above lemma is telling us that either we can increase the size of $U$ or there are many elements which we can choose as $T(i,c)$ while preserving independence of $S_i$. Potentially these choices could cause $C_c$ to be dependent, however this would mean that we can remove some other entry in column $c$. While this will not increase the number of entries in $U$ it offers us a way to ``move'' an empty cell from $(i,b)$ to a position in column $c$. This motivates the following definition, which was essentially implicit in \cite{us} (in the general case of Kahn's basis conjecture it is more convenient to make this definition explicit).

\begin{defn}
Let $T(i,b)=\emptyset$ and $T(j,c)=y'\neq \emptyset$. We say that position $(j,c)$ is $(i,b)$-\textit{removable} if there is an $(i,b)$-addable element $x \in B_{i,c}$ such that $C_c+x-y'$ is independent. 
\end{defn}

\begin{claim}\label{claim:removability}
Let $T(i,b)=\emptyset$. If there are $r$ distinct $(i,b)$-addable elements in $B_{i,c}$ then either one such element is independent of $C_c$ or there are at least $r$ positions in column $c$ which are $(i,b)$-removable. 
\end{claim}
\begin{proof}
Let $X$ denote the set of $(i,b)$-addable elements in $B_{i,c}$, so that $|X|=r$. If $|S|=r>|C_c|$ then by the augmentation property there is an element $x \in S$ which is independent of $C_c$, so let us assume $r \le |C_c|.$ Once again the augmentation property implies that there is a set $Y \subseteq C_c$ such that $|Y|=r$ and $I=X+C_c-Y$ is independent. For any $y \in Y$, taking $J=C_c-y$ and using the augmentation property, there is an element $x\in I \setminus J$ such that $C_c-y+x$ is independent, showing that the position of $y$ is removable.  
\end{proof}

Note again that if the first case of \cref{claim:removability} occurs, meaning that there is an $(i,b)$-addable element $x\in B_{i,c}$ which is independent of $C_c$, then we can increase the size of $U$ by letting $T(i,c)=x$, possibly after a simple swap.

The following claim gives a good illustration of how to use the ideas developed in this section to find many addable elements. It will be very useful later on.

\begin{claim}\label{claim:1-addability} 
Let $T(i,b)=\emptyset$. Then either we can increase the size of $U$ or there are at least $(n-|S_i|)\left(n-|C_b|\right)$ elements which are $\left(i,b\right)$-removable.
\end{claim}
\begin{proof}
By \cref{claim:many-good} there are at least $n-|C_b|$ columns that
are $(i,b)$-swappable. For
each such column $c$, by the augmentation property, there are at least $n-|S_i|$ elements $x\in B_{i,c}$
independent to all the elements of $S_i$. Each of these elements is $\left(i,b\right)$-addable
by \cref{claim:add-if-good}, or else we can increase the size of $U$ directly. For each such addable element, by \cref{claim:removability} we can again either increase the size of $U$ directly or there are $n-|S_i|$ positions which are $(i,b)$-removable. That is to say, there are at least $(n-|S_i|)\left(n-|C_b|\right)$
positions which are $\left(i,b\right)$-removable, as claimed.
\end{proof}

In our proof of \cref{thm:1/2-kahn} we will need the following lemma. It will allow us to ensure that the new removable positions we find which are in the same column are still distinct.

\begin{lem}
\label{lem:matching}Then for each $B_{i,c}$, we
can find an injection $\phi_{i,c}:S_i\to B_{i,c}$ such that for all
$x \in S_i$, $\phi_{i,c}(x)$
is independent of $S_i-x$.
\end{lem}

\begin{proof}
Consider the bipartite graph $G$ where the first part consists of the elements of $S_i$ and
the second part consists of the elements of $B_{i,c}$, with an edge
between $x\in S_i$ and $y\in B_{i,c}$ if $y$ is independent
of $S_{i}-x$. We use Hall's theorem
to show that there is a matching in this bipartite graph covering $S_i$. Indeed,
consider some $W\subseteq S_i$. By the augmentation property, there
are at least $\left|W\right|$ elements $y\in B_{i,c}$ such that $S_i-W+y$
is an independent set, and again using the augmentation property,
each of these can be extended to an independent set of the form $S_i+y-x$
for some $x\in W$. That is to say, $W$ has at least
$\left|W\right|$ neighbours in $G$.
\end{proof}

\subsection{Cascading swaps}

Informally speaking, for any $i_0$ for which $S_{i_0}$ is not a
basis, we have showed that either we can increase the size of $U$, or there are many positions $\left(i_1,c_1\right)\in U$ which we can free up after performing a simple swap. In the latter case, we find a row with as many removable elements as possible and then repeat the argument starting from this row, either finding a way to increase the size of $U$ or finding more positions that we can free up after a sequence of two swaps. We then iterate this argument, continually increasing the number of positions we can free up. This cannot continue for too long, and eventually we will find a way to increase the size of $U$ via a cascading sequence of swaps, as desired.

The next definition makes precise the cascades
that we consider. We remark that the definitions of addability and removability are with respect to a table $T$, and it makes sense to say that an element or position is addable or removable in a different table $T'$.

\begin{defn}
Consider a sequence of distinct rows $i_{0},\dots,i_{\ell-1},i_{\ell}$.
Say a position $\left(i_{\ell},c_{\ell}\right)$
is \emph{cascade-removable }with state $T_\ell$, with respect to $i_{0},\dots,i_{\ell-1}$, if
\begin{itemize}
    \item $\ell=0$ and $T(i_0,c_0)=\emptyset$, and $T_0:=T$, or; 
    \item $\ell>0$, and there is a position $(i_{\ell-1}, c_{\ell-1})$ which is cascade-removable with state $T_{\ell-1}$, with respect to $i_{0},\dots,i_{\ell-1}$, such that $(i_{\ell}, c_{\ell})$ is $(i_{\ell-1}, c_{\ell-1})$-removable in $T_{\ell-1}$. Moreover $T_\ell$ is the result of performing this removal in $T_{\ell-1}.$
\end{itemize}
We will call the sequence of operations resulting in the removal of $(i_\ell,c_\ell)$ in $T_\ell$ a \textit{cascade}, and we require that there is a cascade in which all the columns $c_i$ are distinct. We write $Q\left(i_{0},\dots,i_{\ell-1}\right)$ for the set of all
elements outside $i_{0},\dots,i_{\ell-1}$ which are cascade-removable
with respect to $i_{0},\dots,i_{\ell-1}$.
\end{defn}
We remark that this notion is almost the same as the notion of cascade-addability defined in \cite{us}. The only difference is that the notion of cascade-addability essentially allows for the possibility that a new element can be directly added into $T$, increasing the size of $U$. In the general case of Kahn's basis conjecture, it is more convenient to first consider cascades that can free up positions (keeping the size of $U$ constant), then separately consider the ways to add an element after this operation.

In the next claim, we show that given $i_{0},\dots,i_{\ell-1}$,
we can either increase the size of $U$ or it is possible to choose $i_{\ell}$ in such a way that the number of cascade-removable elements increases.
\begin{claim}
\label{claim:cascade-increase}Consider a sequence of distinct rows
$i_{0},\dots,i_{\ell-1}$ with $1 \le \ell<f$ (recall that $f$ denotes the number of rows of $T$).
Then either we can increase the size of $U$, or we can
choose $i_{\ell}\ne i_{0},\dots,i_{\ell-1}$ such that
\begin{equation}
\left|Q\left(i_{0},\dots,i_{\ell}\right)\right|\ge\frac{\left|Q\left(i_{0},\dots,i_{\ell-1}\right)\right|}{f-\ell}\cdot\left(n-f-\ell\right)-(\ell+1) n.\label{eq:recurrence}
\end{equation}
\end{claim}

\begin{proof}
Let us choose
$i_{\ell}\in [f]\setminus\left\{ i_{0},\dots i_{\ell-1}\right\} $ with
the maximum number of elements of $Q\left(i_{0},\dots,i_{\ell-1}\right)$.
Let $Q=i_{\ell}\cap Q\left(i_{0},\dots,i_{\ell-1}\right)$, so
\[
\left|Q\right|\ge\frac{\left|Q\left(S_{0},\dots,S_{\ell-1}\right)\right|}{f-\ell}.
\]
Apply \cref{lem:matching} to row $i_{\ell}$ to obtain an injection $\phi_{i_\ell,b}$ for every column $b$. We suppress the dependence on $i_\ell$ and just write $\phi_b.$

First we will show that unless we can increase the size of $U$ there are at least $\left|Q\right|\left(n-\ell-f\right)$ cascade-removable positions (with respect to $i_{0},\dots,i_{\ell}$) which are not in any of the rows $i_{0},\dots,i_{\ell}$. We start by showing that for any $\left(i_{\ell},c_{\ell}\right)\in Q$ (which is, by the definition of $Q$, cascade-removable with respect to $i_{0},\dots,i_{\ell-1}$, say with state $T_{\ell-1}$), there are $n-f$ columns $c$ for which $\phi_{c}(i_{\ell},c_{\ell})$ is $(i_\ell,c_\ell)$-addable in $T_{\ell-1}$. This follows from \cref{claim:many-good}, which
implies there are at least $n-\left|C_{c_{\ell}}\right|\ge n-f$ columns which are $\left(i_\ell,c_{\ell}\right)$-swappable, and \cref{claim:add-if-good}, which implies that for each such column $c$, $\phi_{c}\left(\left(i_{\ell},c_{\ell}\right)\right)$
is indeed $\left(i_{\ell},c_{\ell}\right)$-addable. If any of these elements does not appear in $T$ we can add it (possibly after a simple swap) after performing a cascade to obtain $T_{\ell-1}$, thereby increasing the size of $U$. So, we can assume that all $(n-f)$ elements of the form $\phi_{c}\left(\left(i_{\ell},c_{\ell}\right)\right)$ appear in $T$. Now, the cascade that removes $\left(i_{\ell},c_{\ell}\right)$ affects at most $\ell$ columns apart from $c_\ell$, so the remaining columns are equal in $T$ and $T_\ell$. Ignoring the affected columns, we have found at least $(n-f-\ell)$ positions (at locations of) $\phi_{c}\left(\left(i_{\ell},c_{\ell}\right)\right)$ which are cascade-removable with respect to $i_0,\dots,i_\ell$, unless they are in the rows $i_{0},\dots,i_{\ell}$. Considering all $\left(i_{\ell},c_{\ell}\right)\in Q$ gives $|Q|(n-f-\ell)$ such positions, at most $(\ell+1)n$ of which are in the rows $i_0,\dots,i_\ell$.
\end{proof}

Now, we want to iteratively apply \cref{claim:cascade-increase} starting
from some row $i_{0}$, to obtain a sequence $i_{0},i_{1},\dots,i_{h}$.
There are two ways this process can stop: either we find a way to
increase the size of $U$, in which case we are done, or else we
run out of distinct rows (that is, $h=f-1$). We want to show that
this latter possibility cannot occur by deducing from \cref{eq:recurrence}
that the $\left|Q\left(i_{0},\dots,i_{\ell}\right)\right|$ increase
in size at an exponential rate: after logarithmically many steps there
will be so many cascade-removable positions that they cannot all be contained
in $U$, and it must be possible to increase the size of $U$.

A slight snag with this plan is that \cref{eq:recurrence} only yields
an exponentially growing recurrence if the ``initial term'' is rather
large. To be precise, let $C$ (depending on $\varepsilon$) be sufficiently
large such that 
\begin{equation}
C\left(1+\varepsilon/2\right)^{\ell-1}\frac{1}{1-\varepsilon}-\ell-1\ge C\left(1+\varepsilon/2\right)^{\ell}\label{eq:C}
\end{equation}
for all $\ell\ge1$.
\begin{claim}
\label{claim:recurrence-estimate}For $S_{0},\dots,S_{h}$ as above,
suppose that $\left|Q\left(S_{0}\right)\right|\ge Cn$ or $\left|Q\left(S_{0},S_{1}\right)\right|\ge Cn$.
Then, for $0<\ell\le\min\left\{ h,\varepsilon n/4\right\} $, we have
\[
\left|Q\left(S_{0},\dots,S_{\ell}\right)\right|\ge C\left(1+\varepsilon/2\right)^{\ell-1}n.
\]
\end{claim}

\begin{proof}
Let $Q_{\ell}=Q\left(i_{0},\dots,i_{\ell}\right)$. We proceed by
induction. First observe that if $|Q_0| \ge Cn$ then \cref{eq:recurrence} and \cref{eq:C} for $\ell=1$ imply $|Q_1|\ge Cn$, giving us the base case. If $\left|Q_{\ell}\right|\ge C\left(1+\varepsilon/2\right)^{\ell-1}n$
then
\begin{align*}
\left|Q_{\ell+1}\right| & \ge\frac{C\left(1+\varepsilon/2\right)^{\ell-1}n}{f-\ell}\cdot\left(n-f-\ell\right)-(\ell+1) n\\
 & \ge\left(C\left(1+\varepsilon/2\right)^{\ell-1}\frac{\left(n-f-\ell\right)}{f}-\ell-1\right)n\\
 & \ge \left(C\left(1+\varepsilon/2\right)^{\ell-1}\frac{1}{1-\varepsilon}-\ell-1\right)n\\
 & \ge C\left(1+\varepsilon/2\right)^{\ell}n.\tag*{\qedhere}
\end{align*}
\end{proof}
If we could choose $i_{0},i_{1}$ such that $\left|Q\left(i_{0}\right)\right|\ge Cn$
or $\left|Q\left(i_{0},i_{1}\right)\right|\ge Cn$, then we would be done, as in \cite{us}. There may not exist suitable starting rows $i_{0},i_{1}\in\S$, but in the next
section we will show that if at least $\varepsilon n/2$ of the $S_i$
in $S$ are not bases, then it is possible to modify $T$
without changing the size of $U$, in such a way that suitable $i_{0},i_{1}$
exist.

\subsection{Increasing the amount of initial addable elements}

Let us assume there are at least $\varepsilon n/2$ $S_i$ which are not
bases. Recall the choice of $C$ from the previous section, and let $D=2C+4$, so that
$D\left(n-f-1\right)-2n\ge Cn$ for large $n$. We prove the following (for large $n$).

\begin{claim}
\label{claim:many-missing}We can modify $T$ in such a way that at least one of the following holds.
\begin{enumerate}
\item [(a)]The size of $U$ increases;
\item [(b)]the size of $U$ does not change, and there is a row $i_{0}$
missing entries in at least $D$ columns;
\item [(c)]the size of $U$ does not change, and there are now distinct rows
$i_{0},i_{1}$ such that $i_{1}$ contains at least $D$ elements
that are $\left(i_{0},b\right)$-removable.
\end{enumerate}
\end{claim}

This suffices for our proof of \cref{thm:1/2-kahn}; indeed, if row $i_{0}$
is missing entries in at least $D$ columns, then by \cref{claim:cascade-increase}, either we can increase the size of $U$
or there are at least $D\left(n-f\right)\ge Cn$ elements which are
$\left(i_{0},b\right)$-removable, meaning that $\left|Q\left(i_{0}\right)\right|\ge Cn$.
If $i_{1}$ contains at least $D$ elements that are $\left(i_{0},b\right)$-addable,
then in the proof of \cref{eq:recurrence} with $\ell=1$ we have $|Q|\ge D$ so either we can increase the size of $U$ or $\left|Q\left(i_{0},i_{1}\right)\right|\ge D\left(n-f-1\right)-2n\ge Cn$.

Before proceeding to the proof of \cref{claim:many-missing},
we first observe that using \cref{lem:matching} we can modify $T$
to ensure that every row $i$ that is not a basis can
be assigned a distinct column $b_i$ with $T(i,b_i)=\emptyset$. To show this,
we prove the following lemma.
\begin{lem}
\label{lem:different-missing}
We can modify $T$ in such a way that the size of each $S_i$ remains the same, and in such a way that there is a choice of disjoint columns $\{ b_1,\dots,b_f\}$ for which any $S_i$ that is not a basis has no element in column $b_i.$
\end{lem}
\begin{proof}
Suppose for some $i$ that we found distinct columns $b_1,\dots,b_{i-1}$
such that, for all $S_{j}$ which are not bases, no element
of $S_{j}$ is in column $b_{j}$. If $S_i$ is a basis we choose an arbitrary unused column as $b_i.$ Otherwise there is a column, say $c$, such that $T(i,c)$ is empty. Then by \cref{claim:many-good} there are at least $n-|C_c|\ge n-f\ge n/2$ columns which are are $(i,c)$-swappable. At least one of these columns does not appear in $\left\{ b_1,\dots,b_{i-1}\right\}$, since $i-1<f\le n/2$. Let $b$ be such a column and set $b_i=b$. By performing a simple swap we can modify $S_i$ in such a way that the cell $(i,b)$ becomes empty, while preserving the independence conditions. 
\end{proof}

Now we prove \cref{claim:many-missing}.
\begin{proof}[Proof of \cref{claim:many-missing}]
Recall that we are assuming there are at least $\varepsilon n/2$
$S_i$ that are not bases. Let $E$ be the largest
integer such that there are at least $M_{E}=\left(\varepsilon/\left(4D^{2}\right)\right)^{E}n$
$S_i$ missing entries in at least $E$ columns. We may assume $1\le E<D$.
By \cref{lem:different-missing} we may find distinct columns $b_1,\dots,b_f$ such that each $S_i$ which is not a basis has $T(i,b_i)=\emptyset$. We describe a procedure
that modifies $T$ to increase $E$.

We create an auxiliary digraph $G$ on the vertex set $[f]$ as follows.
For every $i$ such that $S_{i}$ is missing entries in at
least $E$ columns, put an arc to $i$ from every $j$ 
such that $S_{j}$ contains at least $E+1$ elements that are $\left(i,b\left(i\right)\right)$-removable.

Say an \emph{$\left(E+1\right)$-out-star} in a digraph is a set of
$E+1$ arcs directed away from a single vertex. Exactly the same proof as in \cite{us} shows that there are $M_{E+1}$ vertex-disjoint $\left(E+1\right)$-out-stars. Now, consider an $\left(E+1\right)$-out-star (with centre $S_{j}$, say). We show how to transfer $E+1$ elements from $S_j$ to its out-neighbours, the end result of which is that $S_j$ is then missing entries in $E+1$ columns. We will then be able to repeat this process for each of our out-stars.

For each of the $E+1$ out-neighbours $S_{i}$ of $S_{j}$ there are
at least $E+1$ positions of $S_j$ which are $\left(i,b_i\right)$-removable. Therefore, for each such $S_i$ we can make a specific choice of such a position, in such a way that each of these $E+1$ choices are \emph{distinct}. For each $S_i$ we can then remove the chosen element from $S_j$ by increasing the size of $S_i$. These modifications will not create any conflicts, because
any addability witness for any element in $S_{0}$ is in a column
unique to that $S_{i}$ (by \cref{lem:different-missing}). After this
operation, $S_i$ is now missing entries in 
at least $E+1$ columns.
\end{proof}

\end{document}